\documentclass[sts,
	reqno,
]{imsart}

\usepackage{amsthm,amsmath, amssymb, color}
\RequirePackage[colorlinks,citecolor=blue,urlcolor=blue]{hyperref}
\usepackage{xcolor,colortbl}
\usepackage{xr}
\RequirePackage{hypernat}

\usepackage{color}
\usepackage{amsfonts, fancybox}
\usepackage{amssymb}
\usepackage[american]{babel}

\usepackage{psfrag,color}
\usepackage{dcolumn}
\usepackage{subcaption}
\usepackage{flafter, bm, dsfont}

\usepackage{multirow}
\usepackage{color}
\usepackage{amsmath}
\usepackage{float}
\usepackage{mathrsfs}
\usepackage{amsfonts}
\usepackage{dsfont}
\usepackage{amssymb}
\usepackage{booktabs}
\usepackage{algpseudocode,algorithm,algorithmicx}
\usepackage{wrapfig}
\usepackage[normalem]{ulem}
\usepackage{siunitx}

\usepackage{amssymb, latexsym}

\usepackage{graphicx}
\usepackage{epstopdf}

\usepackage[sort]{cite}
\usepackage{boxedminipage}
\usepackage{tkz-graph}
\usetikzlibrary{external}
\newtheorem{assumption}{\sc Assumption}
\newtheorem{theorem}{Theorem}

\newtheorem{proposition}{Proposition}
\newtheorem{corollary}{Corollary}
\newtheorem{lemma}{Lemma}

\newtheorem{definition}{Definition}
\newcommand{\cA}{\mathcal{A}}

\newcommand{\cC}{\mathcal{C}}
\newcommand{\cD}{\mathcal{D}}
\newcommand{\cE}{\mathcal{E}}
\newcommand{\cF}{\mathcal{F}}

\newcommand{\cM}{\mathcal{M}}

\newcommand{\cP}{\mathcal{P}}
\newcommand{\cQ}{\mathcal{Q}}

\newcommand{\cS}{\mathcal{S}}

\newcommand{\cX}{\mathcal{X}}
\newcommand{\cY}{\mathcal{Y}}

\newcommand{\RR}{\mathbb{R}}

\newcommand{\1}{\mathds{1}}

\newcommand*{\E}{\mathbb E}

\newcommand*{\p}{\mathbb P}

\newcommand*{\ep}{\varepsilon}
\DeclareMathOperator*{\argmin}{argmin}

\newcommand{\defeq}{:=}
\newcommand{\orlicz}[1]{\|#1\|_{\psi_1}}
\AtBeginDocument{
  \label{CorrectFirstPage}
  
}
\begin{document}

\begin{frontmatter}
\title{Uncoupled isotonic regression via minimum Wasserstein deconvolution}

\runtitle{Uncoupled regression} %

\author{\fnms{Philippe}~\snm{Rigollet}\thanksref{t5}\ead[label = rigollet]{rigollet@math.mit.edu} \and
	\fnms{Jonathan}~\snm{Weed}\thanksref{t6}\ead[label = weed]{jweed@math.mit.edu}
}

\affiliation{Massachusetts Institute of Technology}

\thankstext{t5}{This work was supported in part by grants NSF DMS-1712596, NSF DMS-TRIPODS-1740751, ONR N00014-17-1-2147, a grant from the MIT NEC Corporation, grant 2018-182642 from the Chan Zuckerberg Initiative DAF and the MIT Skoltech Seed Fund.}
\thankstext{t6}{This work was supported in part by NSF Graduate Research Fellowship DGE-1122374.}

\runauthor{Rigollet and Weed}
\maketitle

\begin{abstract}
{Isotonic regression is a standard problem in shape-constrained estimation where the goal is to estimate an unknown nondecreasing regression function $f$ from independent pairs $(x_i, y_i)$ where $\E[y_i]=f(x_i), i=1, \ldots n$. While this problem is well understood both statistically and computationally, much less is known about its uncoupled counterpart where one is given only the unordered sets $\{x_1, \ldots, x_n\}$ and $\{y_1, \ldots, y_n\}$. In this work, we leverage tools from optimal transport theory to derive minimax rates under weak moments conditions on $y_i$ and to give an efficient algorithm achieving optimal rates. Both upper and lower bounds employ moment-matching arguments that are also pertinent to learning mixtures of distributions and deconvolution.}
\end{abstract}

\end{frontmatter}
\section{Introduction}

Optimal transport distances have proven valuable for varied tasks in machine learning, computer vision, computer graphics, computational biology, and other disciplines; these recent developments have been supported by breakneck advances in computational optimal transport in the last few years~\cite{Cut13, AltWeeRig17,PeyCut18,AltBacRud18}. This increasing popularity in applied fields has led to a corresponding increase in attention to optimal transport as a tool for theoretical statistics~\cite{ForHutNit19,RigWee18,ZemPan18}. In this paper, we show how to leverage techniques from optimal transport to solve the problem of \emph{uncoupled isotonic regression}, defined as follows.

Let $f$ be an unknown nondecreasing regression function from $[0, 1]$ to $\RR$, and for $i = 1, \dots, n$, let
\begin{equation*}
y_i = f(x_i) + \xi_i\,,
\end{equation*}
where $\xi_i \sim \cD$ are i.i.d.\ from some known distribution $\cD$ and $x_i$ are fixed (deterministic) design points. We note that the location of the design points is immaterial as long as $x_1 < \dots < x_n$. Given $p\ge 1$, the goal of isotonic regression is to produce an estimator $\hat f_n$ that is close to $f$ in the sense that $\E\|\hat f_n -f\|_p^p$ is small, where for any  $g$  from $[0, 1]$ to $\RR$ we define
\begin{equation}
\label{EQ:defellp}
\|g\|^p_p := \frac 1n \sum_{i=1}^n |g(x_i)|^p\,.
\end{equation}

The key novelty in \emph{uncoupled} isotonic regression is that the data at hand to construct $\hat f_n$ is given by the unordered sets $\{y_1, \dots, y_n\}$ and $\{x_1, \dots, x_n\}$. Informally, one does not know ``which $x$ corresponds to which $y$." In contrast, for  \emph{standard} isotonic regression, estimation is performed on the basis of the \emph{coupled} data $\{(x_1, y_1), \dots, (x_n, y_n)\}$. To our best knowledge, uncoupled isotonic regression was introduced in~\cite{CarSch16} as a natural model for situations arising in the social sciences where uncoupled data is a common occurrence. For instance, the authors of~\cite{CarSch16} give the example of analyzing data collected by two different organizations, such as wage data collected by a governmental agency and housing price data collected by a bank. The relationship between wages and housing prices can naturally be assumed to be monotonic. Though these data sets involve the same individuals, the data is uncoupled, and no paired information exists. Our results indicate that despite the lack of paired data, a relationship between the data sets \emph{can} be learned. In addition to raising obvious privacy issues, this result also has drastic implications for sample sizes, since it suggests that it is possible to integrate extremely large datasets such a census data or public real estate data even in the absence of coupled data.

While standard isotonic regression is a well understood and classical problem in shape-constrained estimation~\cite{Gee90,Mam91,Gee93,RobWriDyk88,MeyWoo00,Zha02,BarBarBre72,NemPolTsy85,BelTsy15,FlaMaoRig16,Bel18}, it is not even clear \emph{a priori} that consistent estimators for its uncoupled version exist. In absence of the noise random variables, $\xi_i, i=1, \ldots, n$, the regression function is easy to estimate using monotonicity: After ordering the sets $\{y_1, \dots, y_n\}$ and $\{x_1, \dots, x_n\}$ as $y_{(1)}\le \ldots \le y_{(n)}$ and $x_{(1)}\le \ldots \le x_{(n)})$, it is clear that $y_{(i)}=f(x_{(i)})$, $i=1, \ldots, n$. In the presence of noise, however, this na\"ive scheme fails, and the problem appears to be much more difficult---see Figure~\ref{FIG:illust}.

\begin{figure}[h]
\begin{center}
\includegraphics[width=.6\textwidth]{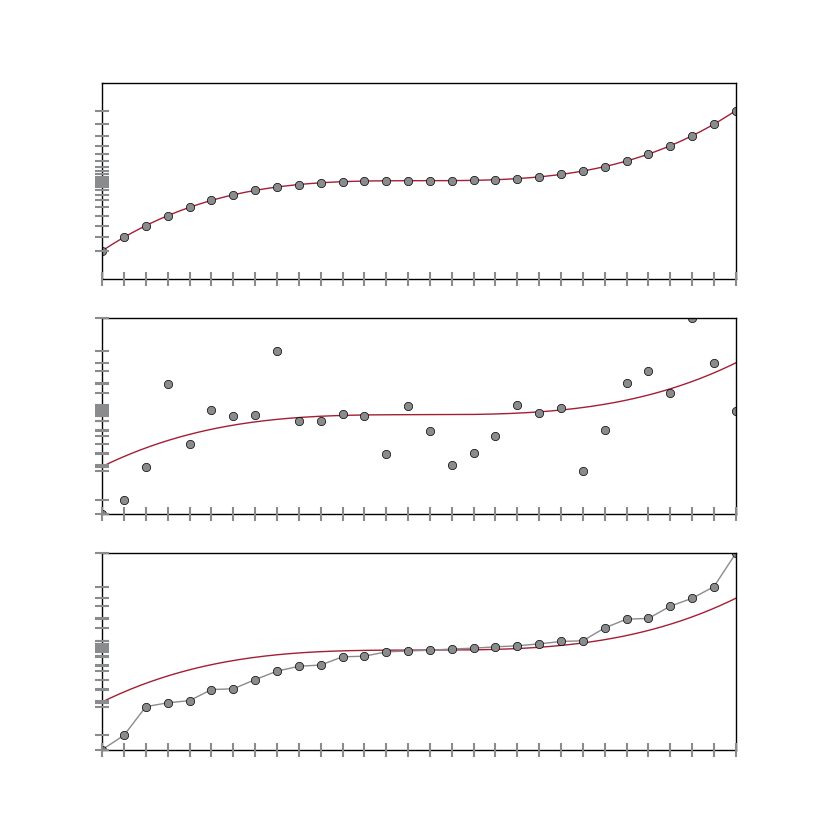}
\caption{In the noiseless case (top figure), either coupled data (gray dots) or uncoupled data (gray tick marks on axes) suffice to recover the regression function (magenta curve). When noise is added (middle figure), uncoupling changes the problem considerably. Estimating the regression function by ordering the sets $\{x_1, \dots, x_n\}$ and $\{y_1, \dots, y_n\}$ does not yield a consistent estimator (bottom figure).
\label{FIG:illust}}
\end{center}
\end{figure}

In this paper, we show that, quite surprisingly, a consistent estimator for $f$ exists under general moment conditions on the noise distribution $\cD$. We define an estimator by leveraging connections with optimal transport and show that it is minimax optimal simultaneously for all choices of $p$ in the performance measure~\eqref{EQ:defellp}. As noted in~\cite{CarSch16}, uncoupled isotonic regression is closely connected to deconvolution, which is a much harder problem than regression from a statistical perspective. Consequently, as our results show, minimax rates for this problem are \emph{exponentially} worse than for the standard isotonic regression problem.  A practical implication is that while uncoupled datasets may be integrated, their size should be exponentially larger in order to lead at least as good statistical accuracy.

\medskip

\noindent{\bf Notation.}
Given quantities $a$ and $b$, we write $a \lesssim b$ to indicate that $a \leq C b$ for some universal constant~$C > 0$ and define $a\vee b \defeq \max(a,b)$.
The notation $a \asymp b$ is used to indicate that $a \lesssim b$ and $b \lesssim a$.
Throughout, $\log$ refers to the natural logarithm, and $\log_+ x := (\log x) \vee 0$.
The terminology ``$\ell_p$ norm" refers always to the empirical $\ell_p$ norm defined in~\eqref{EQ:defellp}.  $\cF$ denotes the class of nondecreasing functions from $[0, 1]$ to $\RR$ and for any $V>0$,  $\cF_V \subset \cF$ denotes the subset of functions $f \in \cF$ such that $|f(x)| \leq V$ for $x \in [0, 1]$.

\subsection{Prior work}
Isotonic regression is a fundamental problem in nonparametric statistics. As such, the literature on this topic is vast and very well established.
A representative result is the following.

\begin{theorem}{\cite{NemPolTsy85}}
If $\cF_V$ is the class of nondecreasing functions from $[0, 1]$ to $\RR$ satisfying $|f(x)| \leq V$ for $x \in [0, 1]$, then
\begin{equation*}
\inf_{g_n} \sup_{f \in \cF_V} (\E \|f - g_n\|_2^2)^{1/2} \asymp \frac{\sigma^{2/3} V^{1/3}}{n^{1/3}}\,,
\end{equation*}
where the infimum is taken over all measurable functions of the data. Moreover the minimax rate is achieved by the least squares estimator over $\cF$ for which efficient algorithms such that the pool-adjacent-violators algorithm are well developed~\cite{RobWriDyk88}.
\end{theorem}

While there are many refinements of this result, the $n^{-1/3}$ rate is a common feature of isotonic regression problems in a variety of contexts.
By contrast, our results indicate that the minimax rate for the uncoupled problem is of order $\frac{\log \log n}{\log n}$.
In other words, the number of samples required to obtain a certain level of accuracy in the uncoupled setting is exponentially larger than the number required for isotonic regression.
This gap illustrates the profound difference between the coupled and uncoupled models.

In~\cite{CarSch16}, the authors propose an estimator of $f$ for uncoupled regression under smoothness assumptions. Crucially, this work draws an important connection between uncoupled isotonic regression and deconvolution and their estimator actually uses deconvolution as a black box. Under smoothness assumptions, rates of convergence may be obtained by combining the results of~\cite{CarSch16} and rates of convergence for the cumulative distribution function (CDF) in deconvolution as in~\cite{DatGolJud11}. Whether the rates obtained in this way would be optimal over smooth classes of functions is unknown, but this question falls beyond the scope of standard shape-constrained estimation. Instead our results show that, as in standard isotonic regression, the function $f$ can be consistently estimated in the uncoupled isotonic regression model, \emph{without smoothness assumptions}. 
Furthermore, we prove matching upper and lower bounds on the optimal rates of estimation with respect to the empirical $\ell_p$ distance, for $1 \leq p < \infty$.

The connection with deconvolution is not hard to see in hindsight: obtaining the function $f$ from the data $\{y_1, \dots, y_n\}$ resembles the problem of obtaining an estimate of a measure $\mu$ on the basis of samples from the convolution $\mu * \cD$, where $\cD$ is a known noise distribution.
As we note below, the metric of interest in our case is the Wasserstein distance between univariate distributions. While question has been recently considered in the deconvolution literature~\cite{CaiChaDed11,DedFisMic15,DedMic13}, our work present the following specificities. First, we make no smoothness assumptions on the noise distribution, and we assume, as is common in the isotonic regression literature, that the regression function $f$ has bounded variation.
This leads to different rates of estimation than those appearing in the deconvolution context.
Moreover, we employ a simple minimum distance estimator (see Section~\ref{sec:ub}) as opposed to the kernel estimators common in deconvolution.
In short, our assumptions, estimator, and results are quite different from those appearing in the deconvolution literature, despite the similarities in the problem setting.

Our techniques leverage moment matching arguments, which have proven powerful in mixture estimation~\cite{MoiVal10,BanRigWee17,WuYan18} and nonparametric statistics~\cite{LepNemSpo99, JudNem02,CaiLow11,CarVer17}.
As in those works, our lower bounds are constructed by leveraging moment comparison theorems, which connect the moments of two distributions to their total variation distance.
Our upper bounds are based on a novel result showing that the Wasserstein distance of any order between univariate measures can be controlled by moment matching.
This result significantly extends and generalizes several similar results in the literature~\cite{KonVal17,WuYan18}.

Finally, it is worth noting that uncoupled isotonic regression bears comparison to a similar problem in which the regression function is assumed to be \emph{linear} instead of isotonic. This model, which goes under the names ``estimation from a broken sample''~\cite{DeGGoe80}, ``shuffled linear regression"~\cite{AbiPooZou17}, ``linear regression without correspondence"~\cite{HsuShiSun17}, ``regression with permuted data''~\cite{PanWaiCou17,SlaBen17}, and ``unlabeled sensing"~\cite{UnnHagVet18}, has been explored from both algorithmic and statistical perspectives. On the algorithmic side, the core question in these works is how to design efficient estimators for multivariate regression problems, which is nontrivial even in the noiseless setting (i.e., when $\xi_i \equiv 0$). On the statistical side, several computationally efficient estimators have been proposed~\cite{PanWaiCou17,AbiPooZou17} with provable guarantees. However, these estimators rely heavily on the linear model and do not extend to the isotonic case. 

\subsection{Model and assumptions}

We focus on the fixed design case, as is common in the literature on isotonic regression.
We assume the existence of a nondecreasing function $f \in \cF_V$ such that
\begin{equation}\label{eqn:model}
y_i = f(x_i) + \xi_i \quad \quad 1 \leq i \leq n\,,
\end{equation}
where $\xi_i \sim \cD$ i.i.d.
We observe the design points $\{x_1, \dots, x_n\}$, which we assume to be distinct, and the (unordered) set of points $\{y_1, \dots, y_n\}$.

We make the following assumptions.

\begin{assumption}
Both $\cD$ and $V$ are known.
\end{assumption}
The assumption that $\cD$ is known is essential and is ubiquitous in the deconvolution literature: if $\cD$ is unknown, then no consistent estimator of $f$ exists.
For instance, if $\cD$ is unknown, then it is impossible to reject the hypothesis that $f$ is identically $0$, and that all the variation in the set $\{y_1, \dots, y_n\}$ is due to noise.
By contrast, the assumption that $V$ is known is for convenience only, since an upper bound on $V$ can be estimated from the data.

We also require that $\cD$ is sub-exponential~\cite{Ver18}, a concept that we define rigorously via Orlicz norms.
\begin{definition}
Let $\psi_1(x):= e^x - 1$.
We define an Orlicz norm $\orlicz{X}$ of a random variable $X$ by
\begin{equation*}
\orlicz{X} \defeq \inf \{ t > 0: \E \psi_1(|X|/t) \leq 1\}\,.
\end{equation*}
We say that a distribution $\cD$ is \emph{sub-exponential} if $\xi \sim \cD$ satisfies $\|\xi\|_{\psi_1}<\infty$ and we write by extension $\|\cD\|_{\psi_1}:=\|\xi\|_{\psi_1}$. 
\end{definition}
It can be shown that $\orlicz{\cdot}$ defines a norm on the space of random variables satisfying $\orlicz{X} < \infty$, and that a random variable has a finite moment generating function in a neighborhood of the origin if and only if $\orlicz{X} < \infty$.
We note also that if $X \in [-V, V]$ almost surely, then $\orlicz{X} \leq 2 V$.

\begin{assumption}\label{assume:sube}
The noise distribution $\cD$ is centered sub-exponential.
\end{assumption}
We note that, in particular, Assumption~\ref{assume:sube} implies that $\cD$ has finite moments of all orders.
Nevertheless, this restriction is quite mild, as this encompasses most distributions which arise in practice and in theory.

Our only use of Assumption~\ref{assume:sube} will be to provide a bound on the moments of $\cD$, which we obtain via the following well known lemma (see, e.g.,~\cite{Ver18}).
We reproduce a proof in Section~\ref{sec:proofs} that exhibits an explicit constant.
\begin{lemma}\label{lem:subexponential_moments}
For all $p \geq 1$,
\begin{equation*}
(\E |X|^{p})^{1/p} \leq p\orlicz{X}\,.
\end{equation*}
\end{lemma}

\subsection{Main results}
Our main results are matching upper and lower minimax bounds for the problem of estimating the regression function in the $\ell_p$ distance, for any $1 \leq p < \infty$.

\begin{theorem}[Upper bound]\label{thm:main_ub}
Assume that $\cD$ is sub-exponential.
There exists an estimator $\hat f_n$ and a universal constant $C$, such that, for all $1 \leq p < \infty$, the risk of $\hat f_n$ over $\cF_V$ satisfies
\begin{equation*}
\sup_{f \in \cF_V} (\E \|f - \hat f_n\|_p^p)^{1/p} \leq  C p V \frac{\log \log n}{\log n}(1 + o_{V, \cD, p}(1))\,,
\end{equation*}
Where $o_{V, \cD, p}(1)$ indicates a quantity depending on $V$, $\orlicz{\cD}$, and $p$ that goes to $0$ as $n \to \infty$.
\end{theorem}
The estimator $\hat f_n$ appearing in Theorem~\ref{thm:main_ub} is a minimum distance estimator with respect to the {Wasserstein distance}, which we call a \emph{minimum Wasserstein deconvolution estimator} (see Section~\ref{sec:ot}).
Surprisingly, the same estimator achieves the above bound for all $1 \leq p < \infty$.
An analysis of this estimator appears in Section~\ref{sec:ub}.

We complement this result with the following lower bound, which holds already in the case when $\cD$ is the standard Gaussian distribution.

\begin{theorem}[Lower bound]\label{thm:main_lb}
Let $\cD = N(0, 1)$.
Under the same conditions as Theorem~\ref{thm:main_ub}, there exists a universal constant $C'$ such that the estimation risk over the class $\cF_V$ satisfies
\begin{equation*}
\inf_{g_n} \sup_{f \in \cF_V} \E \|f - g_n\|_p \geq C' V \frac{\log \log n}{\log n}(1 + o_{V}(1))\,,
\end{equation*}
where the infimum is taken over all measurable functions of the data.
\end{theorem}

The proofs of Theorems~\ref{thm:main_ub} and~\ref{thm:main_lb} rely on hitherto unexplored connections between isotonic regression and optimal transport between probability measures~\cite{Vil08}.
To exploit this connection, we establish a novel result connecting the Wasserstein $p$-distance between two univariate distributions with the differences in the moments of the two distributions (Theorem~\ref{thm:moment_matching_wp}).
Since we believe this connection will prove useful for other works, we prove a more general version than is needed to obtain Theorems~\ref{thm:main_ub} and~\ref{thm:main_lb}.
While similar results have appeared elsewhere in the literature for the $W_1$ distance~\cite{KonVal17,WuYan18}, our general version is the first to our knowledge to apply to $W_p$ for $p > 1$ and to unbounded measures.

\section{Uncoupled regression via optimal transport}\label{sec:ot}
The observation that forms the core of our work is that the uncoupled regression model naturally relates to the Wasserstein distance between univariate measures.

\subsection{Minimum Wasserstein deconvolution}
We first recall the following  definition.

\begin{definition}
For $1 \leq p < \infty$, the \emph{Wasserstein-$p$ distance} between two probability distributions $\mu$ and $\nu$ is defined by
\begin{equation}
\label{eqn:def_wp}
W_p(\mu, \nu) := \inf_{\gamma \in \cC(\mu, \nu)} \left(\int_\infty^\infty |x - y|^p \, \mathrm{d}\gamma(x, y)\right)^{1/p}\,,
\end{equation}
where the infimum is taken over the set $\cC(\mu, \nu)$ of all joint distributions on $\RR \times \RR$ with first marginal $\mu$ and second marginal $\nu$.
\end{definition}

For all $p \ge 1$, the space $\cM$ of probability measures having finite moments of all orders equipped with the distance $W_p$ defines a metric space denoted by $(\cM, W_p)$.
The key observation is that the risk in isotonic regression can be controlled via the Wasserstein distance. To see this, we need the following definition.

\begin{definition}\label{def:pushforward}
Let $x_1, \dots, x_n$ be fixed.
For any nondecreasing function $g: [0, 1] \to \RR$, denote by $\pi_g$ the measure
\begin{equation*}
\frac 1 n \sum_{i=1}^n \delta_{g(x_i)}\,.
\end{equation*}
We call $\pi_g$ a \emph{pushforward measure} (of the uniform measure on $\{x_1, \ldots, x_n\}$ through $g$).
\end{definition}

The following proposition establishes the central connection between isotonic regression functions and the Wasserstein distance.

\begin{proposition}\label{prop:log_is_wasserstein}
Let $\cF$ be the class of nondecreasing functions from $[0, 1]$ to $\RR$.
For all $1 \leq p < \infty$, the map $f \mapsto \pi_f$ is an isometry between $(\cF, \ell_p)$ and $(\cM, W_p)$.
In other words, the empirical $\ell_p$ distance corresponds to the Wasserstein distance between the pushforward measures:
\begin{equation*}
\|f - g\|_p = W_p(\pi_f, \pi_g)\,.
\end{equation*}
\end{proposition}

\begin{proof}
Let $\gamma = \frac 1n \sum_{i=1}^n \delta_{(f(x_i), g(x_i))}$.
Clearly $\|f - g\|_p^p = \int_{-\infty}^\infty |x - y|^p \, \mathrm{d}\gamma(x, y)$, and $\gamma$ is a coupling between $\pi_f$ and $\pi_g$.
It suffices to show that this coupling is optimal in the sense that it realizes the minimum definition~\eqref{eqn:def_wp}.
For $i, j \in [n]$, the monotonicity of $f$ and $g$ implies
\begin{equation*}
(f(x_i) - f(x_j))(g(x_i) - g(x_j)) \geq 0\,.
\end{equation*}
Therefore, the support $\left\{\left(f(x_i),g(x_i)\right), i=1, \ldots, n\right\}$ of $\gamma$ is \emph{monotone}, meaning that for any $(a, b), (c, d) \in \mathrm{supp}(\gamma)$, the implication $a < c \implies b \leq d$ holds.
Standard facts~\cite[Theorem~2.9]{San15} then imply that it is optimal.
\end{proof}

Denote by $\hat \pi$ the empirical distribution of the observation $\{y_1, \dots, y_n\}$. A sample from $\hat \pi$ is marginally distributed as $\pi_f * \cD$, the convolution of the pushforward measure $\pi_f$ with the noise distribution $\cD$.
Thus, finding $\pi_f$ can be viewed as a deconvolution problem, or equivalently as a mixture learning problem whose centers are given by the distribution $\pi_f$. Consequently, our estimator is similar to estimators proposed in the mixture learning literature.
One common choice is to choose the parameter that minimizes the distance to the empirical distribution in the Kolomogorov-Smirnov distance~\cite{DeeKru68,Che95,HeiKah15}; however, Proposition~\ref{prop:log_is_wasserstein} suggests as an estimator a minimizer of $g \mapsto W_p(\pi_g * \cD, \hat \pi)$ over a suitable function class.
Such estimators were introduced in~\cite{BasBodReg06} under the name \emph{minimum Kantorovich distance estimators} and shown to be consistent under regularity assumptions. By analogy, we call our technique \emph{minimum Wasserstein deconvolution}.\footnote{Note that the shorthand \emph{Wassertein deconvolution} has appeared in the deconvolution literature~\cite{DedMic13,DedFisMic15} to refer to deconvolution problems in which the Wasserstein distance is used as a measure of success. We emphasize that \emph{minimum Wasserstein deconvolution} refers here to a novel method to perform deconvolution based on Wasserstein distances. Moreover, in light of Proposition~\ref{prop:log_is_wasserstein}, this method also achieves good performance in the Wasserstein metric.}

We focus on the following estimator:
\begin{equation}\label{eqn:est_def}
\hat f \in \argmin_{g \in \cF_V} W^2_2(\pi_g * \cD, \hat \pi)\,.
\end{equation}

As Theorem~\ref{thm:main_ub} shows, the estimator $\hat f$ is \emph{adaptive} to $p$, in the sense that it converges to $f$ at the same rate in all $\ell_p$ metrics. Furthermore, by Theorem~\ref{thm:main_lb}, this rate is minimax optimal.

The definition of our estimator involves the distance $W_2$. However, our analysis reveals that $W_2$ can be replaced by $W_r$ for any $r \in (1, \infty)$ to obtain an estimator with the same performance.
Indeed, the interested reader may check that the only results which need to be updated are Theorem~\ref{thm:w2_to_moments} and Proposition~\ref{proposition:estimator_expectation}.
Theorem~\ref{thm:w2_to_moments} can be replaced by a similar argument following~\cite[Proposition~2.21, (ii)]{PflPic14}.
Likewise, Proposition~\ref{proposition:estimator_expectation} holds with exactly the same proof, since it relies only on the triangle inequality (which holds for all $W_r$) and Lemma~\ref{lem:empirical_w2}, a more general version of which can be found in~\cite{BobLed16}.

\subsection{A computationally efficient estimator}\label{sec:computational}
\emph{A priori}, it is unclear how to optimize the function $f \mapsto W_2^2(\pi_f * \cD, \hat \pi)$ explicitly. In order to obtain an estimator which can be computed in polynomial time, we propose in this section
a computationally efficient version of~\eqref{eqn:est_def}, which enjoys the same theoretical guarantees.
We first relax~\eqref{eqn:est_def} and consider instead the program
\begin{equation*}
\argmin_{\mu \in \cM_V} W_2^2(\mu * \cD, \hat \pi)\,,
\end{equation*}
where the minimization is taken over \emph{all} measures with support in $[-V, V]$.
This is now a convex program, albeit an infinite dimensional one.
However, we show below that it suffices to optimize over a finite-dimensional subset of $\cM_V$, which yields a tractable convex program.
Finally, we show how to round the resulting solution $\hat \mu$ to a pushforward measure in the sense of Definition~\ref{def:pushforward}.

We first consider the following quantization of the real line.
Assume $n \geq 3$.
Let $\alpha_0 \defeq -  (V + \sigma) \log n $ and set
\begin{equation*}
\alpha_i \defeq \alpha_0 + i \cdot \frac{V + \sigma}{n^{1/4}} \quad \quad \text{for $1 \leq i \leq N \defeq \lceil 2 n^{1/4} \log n \rceil$}\,.
\end{equation*}
Let $\cA \defeq \{\alpha_i\}_{i = 0}^N$, and denote by $\cM_{\cA, V}$ the set of measures supported on $\cA \cap [-V, V]$, which is a discrete set of cardinality $O(n^{1/4})$.
Finally, define the projection operator $\Pi_\cA: \RR \to \cA$ by
\begin{equation*}
\Pi_\cA(x) \defeq \left\{\begin{array}{ll}
\alpha_0 & \text{ if $x < \alpha_0$} \\
\alpha_i & \text{ if $\alpha_i \leq x < \alpha_{i+1}$ for $0 \leq i \leq N-1$} \\
\alpha_N & \text{ if $x \geq \alpha_N$.}
\end{array}\right.
\end{equation*}

We propose the following computationally efficient estimator:
\begin{equation}\label{eqn:efficient}
\hat \mu \in \argmin_{\mu \in \cM_{\cA, V}} W_2^2({\Pi_\cA}_\sharp(\mu * \cD), \hat \pi)\,,
\end{equation}
where ${\Pi_\cA}_\sharp(\mu * \cD)$ is the pushforward of the measure $\mu * \cD$ by the projection operator.
The map $\mu \mapsto W_2^2({\Pi_\cA}_\sharp(\mu * \cD), \hat \pi)$ is convex, and subgradients can be obtained by standard methods in computational optimal transport~\cite{PeyCut18}.
The measure $\hat \mu$ can therefore be obtained efficiently.

In general, the solution $\hat \mu$ to~\eqref{eqn:efficient} will not be of the form $\pi_g$ for some isotonic function $g$.
However, a sufficiently close function can easily be obtained.
Given a measure $\mu$, denote by $\cQ_\mu$ the quantile function of $\mu$; we then define $\hat g$ by
\begin{equation*}
\hat g(x_i) := \cQ_{\hat \mu}(i/n) \quad \quad \text{for $1 \leq i \leq n$}\,,
\end{equation*}
and extend $\hat g$ to other values in $[0, 1]$ arbitrarily so that the resulting function lies in $\cF_V$.

\begin{proposition}\label{prop:efficient_rate}
The estimator $\hat g$ achieves the same rate as the estimator $\hat f$ defined in~\eqref{eqn:est_def}.
\end{proposition}
The proof is deferred to Appendix~\ref{sec:proof:prop:efficient_rate}.

\subsection{From Wasserstein distances to moment-matching, and back}
\label{sec:wass_moments}
Both the upper and lower bounds for the uncoupled regression problem (Theorems~\ref{thm:main_ub} and~\ref{thm:main_lb}) depend on moment-matching arguments that we gather here.
The core of our approach is Theorem~\ref{thm:moment_matching_wp}, which establishes that the Wasserstein distance between univariate measures can be controlled by comparing the moments of the two measures.
In Proposition~\ref{proposition:moment_matching_tight}, we give examples establishing that Theorem~\ref{thm:moment_matching_wp} cannot be improved in general.

Similar moment-matching results for the Wasserstein-1 distance $W_1$ have appeared in other works~\cite{KonVal17,WuYan18}, but in general these results rely on arguments via polynomial approximation of Lipschitz functions combined with the dual representation of $W_1$~\cite{Vil03}.
This approach breaks down for measures with unbounded support, and cannot establish tight bounds $W_p$ for $p > 1$.
By contrast, Theorem~\ref{thm:moment_matching_wp} applies to all measures with convergent moment generating functions, and yields bounds for $W_p$ for all $1 \leq p < \infty$.

\begin{definition}
For any distributions $\mu$ and $\nu$ on $\RR$ and $\ell \geq 1$, define
\begin{equation*}
\Delta_\ell(\mu, \nu) \defeq \left| \E[X^\ell] - \E[Y^\ell]\right|^{1/\ell} \quad \quad X \sim \mu, Y \sim \nu\,.
\end{equation*}
When $\mu$ and $\nu$ are clear from context, we abbreviate $\Delta_\ell(\mu, \nu)$ by $\Delta_\ell$.
\end{definition}

We are now in a position to state the main result of this section: it shows that two distributions with similar moments are close in Wasserstein distance. Its proof is postponed to Appendix~\ref{app:momentmatchingproof}.

\begin{theorem}\label{thm:moment_matching_wp}
Let $\mu$ and $\nu$ be two distributions on $\RR$ whose moment generating functions are finite everywhere.
There exists a universal constant $C > 0$ such that, for $1 \leq p < \infty$,
\begin{equation*}
W_p(\mu, \nu) \leq C p  \sup_{\ell \geq 1} \frac{\Delta_\ell(\mu, \nu)}{\ell}\,.
\end{equation*}
\end{theorem}

Theorem~\ref{thm:moment_matching_wp} includes as a corollary the following result for bounded measures, a version of which appeared in~\cite[Proposition~1]{KonVal17} for the $p = 1$ case.
\begin{corollary}\label{corollary:bounded}
Let $\mu$ and $\nu$ be two measures supported on $[-1, 1]$.
For any $k \geq 1$, if $\max_{\ell \leq k} \Delta_\ell^\ell(\mu, \nu) \leq \ep < 1$, then
\begin{equation*}
W_p(\mu, \nu) \lesssim p \left(\frac{1}{\log(1/\ep)} \vee \frac{1}{k}\right)\,.
\end{equation*}
\end{corollary}
\begin{proof}
For $\ell \leq k$, we have by assumption the bound $\Delta_\ell \leq \ep^{1/\ell}$, whereas for $\ell > k$, we have the bound $\Delta_\ell \lesssim 1$ because $\mu$ and $\nu$ are supported on $[-1, 1]$.
Applying Theorem~\ref{thm:moment_matching_wp} and noting that $\ell \mapsto \ep^{1/\ell}/\ell$ is maximized at $\ell = \log(1/\ep)$ yields the claim.
\end{proof}

Our results imply a similar simple result for sub-Gaussian measures. We state it as a result of independent interest but will not need it to analyze uncoupled isotonic regression.

\begin{corollary}
Let $\mu$ and $\nu$ be two sub-Gaussian measures.
For any $k \geq 1$, if $\max_{\ell \leq k} \Delta_\ell^\ell(\mu, \nu) \leq \ep < 1$, then
\begin{equation*}
W_p(\mu, \nu) \lesssim p \left(\frac{1}{\log(1/\ep)} \vee \frac{1}{\sqrt k}\right)\,.
\end{equation*}
\end{corollary}
\begin{proof}
The proof is the same as the proof of Corollary~\ref{corollary:bounded}, except that we replace the estimate $\Delta_\ell \lesssim 1$ for $\ell > k$ by the estimate $\Delta_\ell \lesssim \sqrt \ell$.
\end{proof}

As the following proposition makes clear, Theorem~\ref{thm:moment_matching_wp} is essentially tight.
\begin{proposition}\label{proposition:moment_matching_tight}
There exists a universal constant $c > 0$ such that, for any $k \geq 1$, there exist two measures $\mu$ and $\nu$ on $[-1, 1]$ such that $\Delta_\ell = 0$ for $1 \leq \ell < k$ but
\begin{equation*}
W_1(\mu, \nu) \geq \frac c k\,.
\end{equation*}
In other words, the dependence on $\sup_{\ell \geq 1} \frac{\Delta_\ell}{\ell}$ cannot be improved.

\medskip

Moreover, there exists a universal constant $c > 0$ such that, for all $\ep > 0$ sufficiently small, there exist two measures $\mu$ and $\nu$ whose moment generating functions are finite everywhere and
\begin{equation*}
W_p(\mu, \nu) \geq c p^{1-\ep} \quad \quad \forall p \geq 1\,.
\end{equation*}
In other words, the dependence on $p$ cannot be improved.
\end{proposition}
A proof of Proposition~\ref{proposition:moment_matching_tight} appears in Appendix~\ref{SEC:proof:proposition:moment_matching_tight}.

\medskip

The following result complements Theorem~\ref{thm:moment_matching_wp} by showing that if two probability measures $\mu$ and $\nu$ are close in Wasserstein-2 distance, then their moments are close. This direction is much easier than that of Theorem~\ref{thm:moment_matching_wp} and illustrates that Wasserstein distances are strong distances.

\begin{theorem}\label{thm:w2_to_moments}
For any two subexponential probability measures $\mu$ and $\nu$ on $\RR$ and any integer $\ell \ge 1$, it holds
\begin{equation*}
\Delta_\ell^\ell(\mu, \nu) \leq (2 \ell)^\ell (\orlicz{\mu } \vee \orlicz{\nu })^{\ell-1} W_2(\mu , \nu )\,.
\end{equation*}
\end{theorem}
\begin{proof}
We employ the following bound~\cite{PflPic14}, Proposition~2.21, (ii), valid for any random variables $X\sim \mu$ and $Y\sim \nu$ and positive integer $\ell$:
\begin{equation*}
\E[X^\ell - Y^\ell] \leq \ell \cdot W_2(\mu, \nu) \Big( \big(\E|X|^{2(\ell-1)}\big)^{1/2} + \big(\E|Y|^{2(\ell-1)}\big)^{1/2}\Big)
\end{equation*}

Lemma~\ref{lem:subexponential_moments} implies
\begin{align*}
\ell\Big((\E|X|^{2(\ell -1)})^{1/2} + (\E|Y|^{2(\ell -1)})^{1/2}\Big) & \leq 2\ell \left(2(\ell - 1)(\orlicz{X}\vee \orlicz{Y})\right)^{\ell - 1} \\%+ \ell (2(\ell - 1)\orlicz{Y})^{\ell - 1} \\
& \leq (2\ell)^\ell(\orlicz{X}\vee \orlicz{Y})^{\ell-1}\,.
\end{align*}
Combining these bounds yields the claim.
\end{proof}

A similar result showing that $\Delta_\ell$ may be controlled by the Wasserstein-1 distance follows directly from the dual representation of $W_1$ as a supremum over Lipschitz functions (see, e.g.,~\cite{Vil03}) when the measures $\mu$ and $\nu$ have \emph{bounded support}. As we will see in the proof of Theorem~\ref{thm:main_ub}, we apply Theorem~\ref{thm:w2_to_moments} to convolved distributions of the form $\mu=\pi_g * \cD$, which have unbounded support whenever $\cD$ does.
Our proof techniques therefore require the use of a stronger metric than $W_1$ whenever the noise distribution~$\cD$ has unbounded support.

\section{Proof of the upper bound}\label{sec:ub}
In this section, we show that the minimum Wasserstein deconvolution estimation~\eqref{eqn:est_def} achieves the upper bound of Theorem~\ref{thm:main_ub}. The proof employs the following steps. 
\begin{enumerate}
\item We show that it follows from the fact that $\hat f$ is a minimum Wasserstein distance estimator that $W_2(\pi_{\hat f} * \cD, \pi_{f} * \cD) $ is small (Proposition~\ref{proposition:estimator_expectation}). 
\item In light of Theorem~\ref{thm:w2_to_moments}, this implies that the sequence $\{\Delta_\ell(\pi_{\hat f} * \cD,\pi_{f} * \cD)\}_{\ell\ge 1}$ is uniformly controlled.
\item A simple lemma (Lemma~\ref{lem:moments_deconvolution}) induces a weaker control for the deconvolved measures so that  $\{\Delta_\ell(\pi_{\hat f},\pi_{f})\}_{\ell \ge 1}$ is also controlled.
\item Finally, we use Theorem~\ref{thm:moment_matching_wp} to control $W_p(\pi_{\hat f},\pi_{f})$ for all $p \ge 1$. 
\end{enumerate}
We collect steps 2--4 into Proposition~\ref{prop:deconvolution}, a deconvolution result which may be of independent interest.

Throughout this section, we assume $\orlicz{\cD} \leq \sigma$.
We first carry out step 1, and show that $\hat f$ satisfies the following ``convolved'' guarantee as a simple consequence of its definition.
\begin{proposition}\label{proposition:estimator_expectation}
The estimator $\hat f$ defined in~\eqref{eqn:est_def} satisfies 
\begin{equation*}
\E W_2(\pi_{\hat f} * \cD, \pi_{f} * \cD) \lesssim (\sigma + V) n^{-1/4}\,.
\end{equation*}
\end{proposition}
\begin{proof}
The triangle inequality and the definition of $\hat f$ imply
\begin{equation}
\label{EQ:pr:prop_est1}
W_2(\pi_{\hat f} * \cD, \pi_{f} * \cD) \leq W_2(\pi_{\hat f} * \cD, \hat \pi) + W_2(\pi_f * \cD, \hat \pi) \leq 2W_2(\pi_{f} * \cD, \hat \pi)\,.
\end{equation}
By definition, the support of $\hat \pi$ is $\{f(x_1) + \xi_1, \dots, f(x_n) + \xi_n\}$.
Let $w_1, \dots, w_n$ be i.i.d.\ samples from $\pi_f$, independent of all other randomness, and denote by $w_{(1)}, \dots, w_{(n)}$ their increasing rearrangement. Since $\{w_i\}$ and $\{\xi_i\}$ are independent, the set $\{w_{(i)} + \xi_i, i=1, \ldots, n\}$ comprises i.i.d.\ samples from $\pi_f * \cD$. Applying the triangle inequality, we get that
\begin{equation}
\label{EQ:pr:prop_est2}
\E [W_2(\pi_{f} * \cD, \hat \pi)] \leq \E[W_2(\pi_f * \cD, \frac1n\sum_{i=1}^n \delta_{w_{(i)} + \xi_i})]+  \E[W_2(\hat \pi, \frac1n\sum_{i=1}^n \delta_{w_{(i)} + \xi_i})]\,.
\end{equation}
It follows from Lemma~\ref{lem:empirical_w2} that 
\begin{equation}
\label{EQ:pr:prop_est3}
\E[W_2(\pi_f * \cD, \frac1n\sum_{i=1}^n \delta_{w_{(i)} + \xi_i})] \lesssim \frac{\sigma + V}{n^{1/4}}\,.
\end{equation}
We now control the second term in the right-hand side of~\eqref{EQ:pr:prop_est2}. A simple coupling between the two measures $\hat \pi$ and $\frac1n\sum_{i=1}^n \delta_{w_{(i)} + \xi_i}$ yields
\begin{align*}
W_2^2(\hat \pi, \frac1n\sum_{i=1}^n \delta_{w_{(i)} + \xi_i}) & \leq \frac1n\sum_{i=1}^n |f(x_i) + \xi_i - w_{(i)} - \xi_i|^2  = \frac1n\sum_{i=1}^n |f(x_i) - w_{(i)}|^2 = W_2^2(\pi_f, \frac 1n \sum_{i=1}^n \delta_{w_i})\,.
\end{align*}
Thus, applying again, Lemma~\ref{lem:empirical_w2}, we get
\begin{equation}
\label{EQ:pr:prop_est4}
\E[W_2(\hat \pi, \frac1n\sum_{i=1}^n \delta_{w_{(i)} + \xi_i})]\lesssim \frac{V}{n^{1/4}}
\end{equation}
Combining~\eqref{EQ:pr:prop_est1}--\eqref{EQ:pr:prop_est4} completes the proof.
\end{proof}

The following uses steps 2--4 to obtain a deconvolution result.
It implies that a bound on $W_2(\mu * \cD, \nu * \cD)$ can yield a bound on $W_p(\mu, \nu)$ for all $p \in [1, \infty)$, as long as $\mu$ and $\nu$ have bounded support.

\begin{proposition}\label{prop:deconvolution}
If $\mu$ and $\nu$ have support lying in~$[-V, V]$ and $W_2(\mu * \cD, \nu * \cD) \leq  (\sigma+V)e^{-e^{1/2}}$, then
\begin{equation*}
W_p(\mu, \nu) \lesssim p V \frac{\log \log \frac{\sigma + V}{W_2(\mu * \cD, \nu * \cD)} + \log_+ \frac{\sigma(\sigma + 2V)}{V}}{\log \frac{\sigma + V}{W_2(\mu * \cD, \nu * \cD)}}\,.
\end{equation*}
\end{proposition}
\begin{proof}
As mentioned above, this proofs goes via a moment-matching argument. Since $\mu$ and $\nu$ have bounded support, their moment generating functions converge everywhere; hence, Theorem~\ref{thm:moment_matching_wp} implies that it suffices to control $\sup_{\ell \geq 1} \frac{\Delta_\ell(\mu, \nu)}{\ell}$ to obtain a bound on $W_p(\mu, \nu)$.

Define $$\ep \defeq \frac{W_2(\mu * \cD, \nu * \cD)}{\sigma+V}\,.$$ 
Note that $\|\mu * \cD\|_{\psi_1} +\|\nu * \cD\|_{\psi_1} \le 2(\sigma+2V)$,
so that Theorem~\ref{thm:w2_to_moments} yields
\begin{equation}
\label{EQ:delta_convolved}
\Delta_\ell^\ell(\mu * \cD, \nu * \cD) \leq \big(4 \ell(\sigma+2V)\big)^\ell\ep
\end{equation}
We now use the following deconvolution Lemma. Its proof is postponed to Appendix~\ref{SEC:proof:lem:moments_deconvolution}.
\begin{lemma}
\label{lem:moments_deconvolution}
For any two subexponential probability measures $\mu$ and $\nu$ on $\RR$ and any integer $\ell \ge 1$, it holds
\begin{equation*}
\Delta_\ell^\ell(\mu, \nu) \leq (4 \ell \orlicz{\cD})^\ell \cdot \sup_{m \leq \ell} \Delta^m_m(\mu * \cD, \nu * \cD)\,.
\end{equation*}
\end{lemma}
Together with~\eqref{EQ:delta_convolved}, it yields
\begin{equation}
\label{EQ:delta_not_convolved}
\Delta_\ell(\mu, \nu) \leq 16 \ell^2  \sigma (\sigma+2V) \ep^{1/\ell}\,, \qquad \ell \geq 1
\end{equation}
We now split the analysis into small and large $\ell$. Assume first that $$\ell < \frac{\log (1/\ep)}{2\log \log (1/\ep) + \log_+ \frac{\sigma(\sigma + 2V)}{V}}\,,$$
Then,~\eqref{EQ:delta_not_convolved} yields
\begin{equation*}
\frac{\Delta_\ell}{\ell} \leq 16 \ell \sigma (\sigma+ 2V) \ep^{1/\ell} \leq 16 \ell \frac{V}{(\log(1/\ep))^2}  \lesssim V \frac{\log \log (1/\ep) + \log_+ \frac{\sigma(\sigma + V)}{V}}{\log (1/\ep)}\,,
\end{equation*}
where we have used the fact that $2 \log \log(1/\ep) + \log_+ \frac{\sigma(\sigma + 2V)}{V} \geq  2 \log \log(1/\ep) \geq 1$, by assumption.
Next assume that 
$$
\ell \geq \frac{\log (1/\ep)}{2\log \log (1/\ep) + \log_+ \frac{\sigma(\sigma + 2V)}{V}}\,.
$$
Since $\mu$ and $\nu$ have bounded support, clearly $\Delta_\ell(\mu, \nu) \lesssim V$ for all $\ell \geq 1$.
Therefore, 
\begin{equation*}
\frac{\Delta_\ell(\mu, \nu)}{\ell} \lesssim V \frac{\log \log (1/\ep) + \log_+ \frac{\sigma(\sigma + 2V)}{V}}{\log (1/\ep)}\,.
\end{equation*}

Combining small and large $\ell$, we obtain
\begin{equation*}
\sup_{\ell \geq 1} \frac{\Delta_\ell(\mu, \nu)}{\ell} \lesssim V \frac{\log \log (1/\ep) + \log_+ \frac{\sigma(\sigma + V)}{V}}{\log (1/\ep)}\,.
\end{equation*}
The proof of Proposition~\ref{prop:deconvolution} then follows by applying Theorem~\ref{thm:moment_matching_wp}.
\end{proof}

We are now in a position to conclude the proof of the upper bound in Theorem~\ref{thm:main_ub}.
Let $W \defeq W_2(\pi_{\hat f} * \cD, \pi_f * \cD)$.
Assume that $n$ is large enough that $n^{1/8} \geq e^{e^{1/2}}$.
Denote by $\cE$ the event on which the inequality $W \leq (\sigma+V)n^{-1/8}$ holds.
\begin{align*}
(\E \|f - \hat f\|_p^p)^{1/p} & = (\E W_p^p(\pi_f, \pi_{\hat f}))^{1/p} \\
& \leq (\E [W_p^p(\pi_f, \pi_{\hat f}) \1_{\cE}])^{1/p} + (\E [W_p^p(\pi_f, \pi_{\hat f}) \1_{\cE^C}])^{1/p}\,.
\end{align*}
On $\cE$, Proposition~\ref{prop:deconvolution} yields
\begin{equation*}
W_p(\pi_f, \pi_{\hat f}) \lesssim p V \frac{\log \log n^{1/8} + \log_+ \frac{\sigma(\sigma + 2V)}{V}}{\log n^{1/8}}\,.
\end{equation*}
On the other hand, since $f, \hat f \in \cF_V$, we have the trivial bound $W_p^p(\pi_f, \pi_{\hat f}) \leq (2 V)^p$, so Markov's inequality combined with Proposition~\ref{proposition:estimator_expectation} yields
\begin{equation*}
(\E [W_p^p(\pi_f, \pi_{\hat f}) \1_{\cE^C}])^{1/p} \leq 2 V \p[W > (\sigma+V)n^{-1/8}]^{1/p} \lesssim V n^{-1/8p}\,.
\end{equation*}
We obtain
\begin{align*}
(\E \|f - \hat f\|_p^p)^{1/p} & \lesssim p V \frac{\log \log n^{1/8} + \log_+ \frac{\sigma(\sigma + 2V)}{V}}{\log n^{1/8}} + V n^{-1/8p} \\
& = p V \frac{\log \log n}{\log n}(1 + o_{V, \sigma, p}(1))\,.
\end{align*}

\section{Proof of the lower bound}\label{sec:lb}
In this section, we prove Theorem~\ref{thm:main_lb}. To that end, we employ the ``method of fuzzy hypotheses''~\cite{Tsy09} and define two prior probability distributions on the space of nondecreasing functions.

Our construction is based on the following lemma which has appeared before in the moment-matching literature.
\begin{lemma}{\cite{CaiLow11, WuYan18}}
\label{lem:priors}
There exists a universal constant $c$ such that, for any $k \geq 1$, there exist two centered probability distributions $P$ and $Q$ on $[-V, V]$%
such that
\begin{equation*}
\Delta_\ell(P,Q)=0 \quad \text{for $\ell = 1, \dots, k-1$,}
\end{equation*}
and such that $W_1(P, Q) \geq \frac{c V}{k}$. %
\end{lemma}

\begin{proof}[Proof of Theorem~\ref{thm:main_lb}]
By the monotonicity of $\ell_p$ norms, it suffices to prove the claim for $p = 1$.
First, we show how a measure on $[-V, V]$ can be reduced to a sample $\{y_1, \dots y_n\}$ from an uncoupled regression model, with a possibly random regression function. 

Let $\mu$ be any measure on $[-V, V]$. Let $Z_1, \dots, Z_n$ be i.i.d\ from $\mu$, and denote by $\{Z_{(i)}\}$ the sorted version of $\{Z_i\}$ such that $Z_{(1)}\le\dots\le Z_{(n)}$. Let $F=F_{Z_1, \ldots, Z_n}$ from $[0, 1]$ to $[-V, V]$ be a \emph{random} monotonically non decreasing function such that
\begin{equation*}
F(x_{(i)}) = Z_{(i)}\,.
\end{equation*}
Finally, let $y_i = Z_i + \xi_i$ where $\xi_i$ are i.i.d $N(0,1)$, and let the pair of unordered sets $\cX=\{x_1, \ldots, x_n\}$, $\cY = \{y_1, \dots, y_n\}$ be the uncoupled observations: The $y_i$'s are i.i.d.\ from $\mu*N(0,1)$ and we denote by $\p_F$ their joint distribution. 
Similarly, denote by $\p_{\hspace{-.2ex} f}$ the joint distribution of $y_1, \ldots, y_n$ when $y_i=f(x_i)+\xi_i$ and note that $\p_f$ need not be a \emph{product} distribution: it is, in general, different from $\big(\pi_f*N(0,1) \big)^{\otimes n}$. This is because the sampling mechanism of uncoupled isotonic regression that does not allow for replacement when sampling from the $x_i$'s. 

Let $\tilde f$ be any measurable function of $y_1, \ldots, y_n$. Fix a $k$ to be chosen later, and let $P$ and $Q$ be the two distributions from Lemma~\ref{lem:priors}.
Then for any $r_n>0$, recalling that $F$ is a random function since it depends on $Z_1, \ldots, Z_n$, it holds
\begin{align}
\sup_{g} \p_{\hspace{-.2ex} g}\left(W_1(\pi_{g}, \pi_{\tilde f})  > r_n\right) &\ge \max\Big\{\int \p_F\big(W_1(\pi_{F}, \pi_{\tilde f})>r_n\big)d P^{\otimes n}(Z_1, \dots, Z_n),\nonumber \\
& \phantom{\ge \max\Big\{}\int \p_F\big(W_1(\pi_{F}, \pi_{\tilde f})>r_n\big)d Q^{\otimes n}(Z_1, \dots, Z_n)\Big\}\,,\label{eq:fuzzy}
\end{align}
where the supremum is taken over all non-decreasing functions $g$ from $[0, 1]$ to $[-V, V]$.

Observe first that the two mixture distributions that appear above are, in fact, product distributions: for any event $\cA$ in the sigma-algebra generated by $y_1, \dots, y_n$,
$$
\int \p_F(\cA) dP^{\otimes n}(Z_1, \dots, Z_n)=P_*^{\otimes n}(\cA)\qquad \text{and}\qquad \int \p_F(\cA) dQ^{\otimes n}(Z_1, \dots, Z_n)=Q_*^{\otimes n}(\cA)\,,
$$
where $P_*=P*N(0,1)$ and $Q_*=Q*N(0,1)$.

For any measure $\mu$ on $[-V, V]$, note that
$$
\pi_F=\frac{1}{n}\sum_{i=1}^n\delta_{Z_i}\,,
$$ 
where the $Z_i$s are i.i.d from $\mu$. Thus by~\cite[Theorem~3.2]{BobLed16} we obtain
\begin{equation*}
\int W_1(\pi_F, \mu) d\mu^{\otimes n}(Z_1, \dots, Z_n) \leq \frac{V}{\sqrt n}\,,
\end{equation*}
which yields via Markov's inequality and the triangle inequality that
\begin{align*}
\int \p_F (W_1(\pi_F, \pi_{\hat f}) > r_n) d\mu^{\otimes n} & \geq \int \p_F(W_1(\mu, \pi_{\hat f}) > 2 r_n) d\mu^{\otimes n} - \int \1\{(W_1(\pi_F, \mu) > r_n)\} d\mu^{\otimes n} \\
& \geq \int \p_F(W_1(\mu, \pi_{\hat f}) > 2 r_n) d\mu^{\otimes n} - \frac{V}{ r_n \sqrt n}\,.
\end{align*}

Next, if $W_1(P,Q)\ge 4r_n$, we get from the triangle inequality that
$$
\p_F\big(W_1(Q, \pi_{\tilde f})>2r_n\big) \ge \p_F\big(W_1(P, \pi_{\tilde f})\le 2r_n\big)\,.
$$
Combining~\eqref{eq:fuzzy} with the above two displays yields
\begin{align*}
\sup_{g} \p_{\hspace{-.2ex} g}\left(W_1(\pi_{g}, \pi_{\tilde f})  > r_n\right) &\ge \max\left\{P_*^{\otimes n}\big(W_1(P, \pi_{\tilde f})>2r_n\big),Q_*^{\otimes n}\big(W_1(P, \pi_{\tilde f})\le 2r_n\big)\right\} - \frac{2V}{r_n \sqrt n}\\
&=\frac{1}{2}\left(1-\mathrm{TV}\left(P_*^{\otimes n},Q_*^{\otimes n}\right)\right) - \frac{2V}{r_n \sqrt n}\,.
\end{align*}

By Lemma~\ref{lem:chi_square}, we have that
\begin{equation*}
\mathrm{TV}(P_*^{\otimes n}, Q_*^{\otimes n})^2 \leq \left(1+ e^{5V^2/2}(eV^2/k)^k\right)^n - 1\,.
\end{equation*}
Choosing $k = c_1\frac{\log n}{\log \log n}$ and $r_n = c_2 V/k$ for suitable constants $c_1$ and $c_2$, we obtain
\begin{equation*}
\sup_{g} \p_{\hspace{-.2ex} g}\left(W_1(\pi_{g}, \pi_{\tilde f}) > c_1^{-1}c_2 V\frac{\log \log n}{\log n}\right) \geq \frac 12 - o_V(1)\,,
\end{equation*}
and the claim follows.
\end{proof}

\section{Conclusion}
Our results establish that uncoupled isotonic regression can surprisingly be solved without further assumptions on the regression function $f$. However, as in nonparametric deconvolution, minimax rates are much slower than standard isotonic regression.
One conclusion of the mixture learning literature is that significantly better results are possible when the original measure has small support~\cite{HeiKah15,WuYan18}.
In the context of uncoupled regression, this suggests that better rates may be available when the regression function $f$ is piecewise constant with a small number of pieces, an assumption which also improves rates of estimation under the standard isotonic regression model~\cite{BelTsy15}. Additional smoothness assumptions or more restrictive shape constraints may also lead to better rates. We leave this question to future work.

In this work, we have restricted ourselves to the univariate problem.
Recent work~\cite{HanWanCha17} has considered the generalization of isotonic regression in which the regression function is a coordinate-wise nondecreasing function on $[0, 1]^d$.
Extending our results to the multidimensional setting is another interesting future research direction.

\section*{Funding}
This work was supported by the National Science Foundation [DGE-1122374 to J.W., DMS-1712596, TRIPODS-1740751, IIS-1838071 to P.R.]; The Office of Naval Research [N00014-17-1-2147 to P.R.]; the Chan Zuckerberg Initiative DAF, a donor advised fund of the Silicon Valley Community Foundation [grant number 2018-182642 to P.R.]; the MIT Skoltech Program [MIT Skoltech Seed Fund to P.R.]; and the Josephine de K\'arm\'an Fellowship Trust [Fellowship to J.W.].

\section*{Acknowledgements}
The authors wish to thank Alexandra Carpentier and Yihong Wu for discussions related to an early version of this work.

\appendix
\section{Omitted proofs}\label{sec:proofs}
\subsection{Proof of Proposition~\ref{prop:efficient_rate}}
\label{sec:proof:prop:efficient_rate}
In the following proof, the symbol $C$ will represent a universal constant whose value may change from line to line.
We will show that the estimator $\hat g$ satisfies
\begin{equation*}
W_2(\pi_{\hat g} * \cD, \hat \pi) \leq W_2(\pi_{\hat f} * \cD, \hat \pi) + C (\sigma + V)n^{-1/4}\,.
\end{equation*}
Following the proof of Proposition~\ref{proposition:estimator_expectation}, this implies
\begin{equation*}
\E W_2(\pi_{\hat g} * \cD, \pi_f * \cD) \leq 2\E W_2(\pi_f * \cD, \hat \pi) +  C (\sigma + V)n^{-1/4} \lesssim (\sigma + V)n^{-1/4}\,,
\end{equation*}
or, in other words, that $\hat g$ satisfies the same inequality as $\hat f$ does (Proposition~\ref{proposition:estimator_expectation}), up to constants.
Since the inequality in Proposition~\ref{proposition:estimator_expectation} is the only fact about $\hat f$ used in the proof of the upper bound, this will serve to establish the claim.

We first analyze the solution $\hat \mu$ to~\eqref{eqn:efficient}.
Let $\hat \nu \in \argmin_{\nu \in \cM_V} W_2^2(\nu * \cD, \hat \pi)$, where the minimization is taken over the set $\cM_V$ of all measures on $[-V, V]$ rather than over the set $\cM_{\cA, V}$.
By Lemma~\ref{lem:disc_close}, there exists a $\hat \nu' \in \cM_{\cA, V}$ such that
\begin{equation}
W_2(\hat \nu', \hat \nu) \leq (V+\sigma)n^{-1/4}\,.\label{eq:disc_close}
\end{equation}
Moreover, by Lemma~\ref{lem:proj_close}, we have that for all $\mu \in \cM_{\cA, V}$,
\begin{equation}
W_2({\Pi_\cA}_\sharp(\mu * \cD), \mu * \cD) \leq C (V+\sigma)n^{-1/4}\,.\label{eq:proj_close}
\end{equation}

Combining these inequalities yields
\begin{alignat*}{3}
W_2(\hat \mu * \cD, \hat \pi) & \leq W_2({\Pi_\cA}_\sharp(\hat \mu * \cD), \hat \pi) + C(V+\sigma)n^{-1/4} \quad \quad \quad && \text{(triangle inequality and~\eqref{eq:proj_close})}\\
& \leq W_2({\Pi_\cA}_\sharp(\hat \nu' * \cD), \hat \pi) + C(V+\sigma)n^{-1/4} && \text{(optimality of $\hat \mu$)}\\
& \leq W_2(\hat \nu' * \cD, \hat \pi) + C(V+\sigma)n^{-1/4}&&\text{(triangle inequality and \eqref{eq:proj_close})}\\
& \leq W_2(\hat \nu * \cD, \hat \pi) + C(V + \sigma)n^{-1/4} &&\text{(triangle inequality and~\eqref{eq:disc_close})}\\
& \leq W_2(\pi_{\hat f} * \cD, \hat \pi) + C(V + \sigma)n^{-1/4}\,,&&\text{(optimality of $\hat \nu$)}
\end{alignat*}
where we have used in the fourth step the fact that, for \emph{any} two measures $\alpha$ and $\beta$,
\begin{equation*}
W_2(\alpha * \cD, \beta * \cD) \leq W_2(\alpha, \beta)\,.
\end{equation*}
(See, e.g.,\cite{San15}, Lemma~5.2.)

Finally, by Lemma~\ref{lem:round_close}, we have
\begin{equation*}
W_2(\hat \mu, \pi_{\hat g}) \lesssim 2V n^{-1/2}\,.
\end{equation*}

Therefore, by another application of the triangle inequality, we obtain
\begin{equation*}
W_2(\pi_{\hat g} * \cD, \hat \pi) \leq W_2(\hat \mu * \cD, \hat \pi) + 2 V n^{-1/2} \leq W_2(\pi_{\hat f} * \cD, \hat \pi) + C(V + \sigma)n^{-1/4}\,,
\end{equation*}
as claimed.

\subsection{Proof of Theorem~\ref{thm:moment_matching_wp}}
\label{app:momentmatchingproof}
The proof of Theorem~\ref{thm:moment_matching_wp} depends on convolving the measures $\mu$ and $\nu$ with a kernel with specific smoothness and decay properties.
This kernel is related to the well-known sinc kernel~\cite{Tsy09}, and coincides with a kernel proposed for deconvolution with respect to Wasserstein distance~\cite{DedMic13}.

We define the kernel as follows.
For any positive integer $m$, denote by $\cS_m$ the distribution on $\RR$ with density function
\begin{equation}\label{eq:sinc_kernel}
f_m(t) := \left\{\begin{array}{ll}
C_m \left(\frac{\sin(t/4 \mathrm{e} m)}{t/4 \mathrm {e} m}\right)^{2 m} & \text{if $t \neq 0$,} \\
C_m & \text{if $t = 0$,}
\end{array}\right.
\end{equation}
where $C_m$ is positive a constant chosen so that $\int_{-\infty}^\infty f_m(t) \, \mathrm{d}t  = 1$.
Lemma~\ref{lem:cm} establishes $C_m \leq 1$.

We require two properties of the distribution $\cS_m$: 
\begin{itemize}
\item[(i)]that it possesses sufficiently many moments, and 
\item[(ii)] that the successive derivatives of the density $f_m$ decay sufficiently quickly.
\end{itemize}

To see that (i) holds, note that for any $p \le 2m-2$, since $f_m(t) \leq 1 \wedge (t/4 \mathrm{e} m)^{-2m}$, it holds
\begin{equation}
\label{EQ:moments_S}
\E[|S|^p] = 2 \int_0^\infty t^p f_m(t) \, \mathrm{d}t  \leq 2 \int_0^{4 \mathrm{e} m} t^p \, \mathrm{d}t + 2 (4 \mathrm{e} m)^{2m} \int_{4 \mathrm{e} m}^\infty t^{p - 2m} \, \mathrm{d}t  \leq 4 (4 \mathrm{e} m)^{p + 1}\leq 4 (4 \mathrm{e} m)^{p+1}\,,
\end{equation}

The requirement~(ii) on the successive derivatives of $f_m$ is ensured by the following lemma. Its proof is presented in Appendix~\ref{sec:lem:derivative_bound}.
\begin{lemma}\label{lem:derivative_bound}
The function $f_m$ is analytic on $\RR$ and satisfies
\begin{equation*}
|f^{(n)}_m(t)| \leq \frac{(8 \mathrm{e} m)^{2m}}{(2 \mathrm{e})^n (4 \mathrm{e} m+|t|)^{2m}} \quad \quad \forall n \geq 0\,.
\end{equation*}
\end{lemma}

With these two lemmas, we can establish the claimed result.

\begin{proof}[Proof of Theorem~\ref{thm:moment_matching_wp}]
The assumption that $\mu$ and $\nu$ have finite moment generating functions implies that $\sup_{\ell \geq 1} \frac{\Delta_\ell}{\ell} < \infty$.
Since the statement of the theorem is scale-invariant, it suffices to prove the claim in the case that
$\sup_{\ell \geq 1} \frac{\Delta_\ell}{\ell} = 1$,
where the claimed bound simplifies to $W_p(\mu, \nu) \lesssim p$.
Also, because $W_q \leq W_p$ for $q \leq p$, we can assume without loss of generality that $p$ is a positive even integer.

Set $m \defeq p/2 + 1$.
Let $X \sim \mu$ and $Y \sim \nu$, and let $S \sim \cS_m$ be independent of $X$ and $Y$.
We denote by $\tilde \mu$ the distribution of $X + S$ and by $\tilde \nu$ the distribution of $Y + S$.

By the triangle inequality applied to $W_p$,
\begin{equation}\label{eq:wp_triangle}
W_p(\mu, \nu) \leq W_p(\mu, \tilde \mu) + W_p(\nu, \tilde \nu) + W_p(\tilde \mu, \tilde \nu)\,.
\end{equation}
Since $(X, X + S)$ is a valid coupling between $\mu$ and $\tilde \mu$, by~\eqref{EQ:moments_S}, it holds
\begin{equation}\label{eq:bias_term}
W_p(\mu, \tilde \mu) \leq \left(\E | X - X - S|^p\right)^{1/p} \leq (4(4 \mathrm{e} m)^{p+1})^{1/p} \lesssim p\,.
\end{equation}

It remains to bound the final term.
Denote by $f_{\tilde \mu}$ and $f_{\tilde \nu}$ the densities of $\tilde \mu$ and $\tilde \nu$, respectively.
By~\cite{Vil08}, Theorem~6.15,
\begin{equation*}
W_p^p(\tilde \mu, \tilde \nu) \leq 2^{p-1} \int_{-\infty}^\infty |t|^p |f_{\tilde \mu}(t) - f_{\tilde \nu}(t)| \, \mathrm{d}t\,.
\end{equation*}

The definitions of $f_\mu$ and $f_\nu$ imply
\begin{align*}
f_{\tilde \mu}(t) - f_{\tilde \nu}(t) & = \E \left[f_m(t-X) - f_m\left(t - Y\right)\right]   = \E \sum_{\ell = 1}^\infty \frac{f_m^{(\ell)}(t) (X^\ell - Y^\ell)}{\ell!}  = \sum_{\ell=1}^\infty \frac{f_m^{(\ell)}(t) \E (X^\ell - Y^\ell)}{\ell!}\,,
\end{align*}
where in the last step we used Fubini's theorem since $\mu$ and $\nu$ have moment generating functions that are finite everywhere.
By applying successively the assumption that $\sup_{\ell \geq 1} \Delta_\ell/\ell \leq 1$, Lemma~\ref{lem:derivative_bound}, and Stirling's approximation, we obtain
\begin{align*}
|f_{\tilde \mu}(t) - f_{\tilde \nu}(t)| \leq  \sum_{\ell=1}^\infty \frac{|f_m^{(\ell)}(t)| \ell^\ell}{\ell!} \leq \frac{(8 \mathrm em)^{2m}}{(4 \mathrm e m + |t|)^{2m}} \sum_{\ell=1}^\infty \frac{\ell^\ell}{(2 \mathrm e)^\ell \ell!} \leq \frac{(8 \mathrm em)^{2m}}{(4 \mathrm e m + |t|)^{2m}}\,.
\end{align*}

Therefore, recalling that $2m = p + 2$, we obtain
\begin{align}
W_p^p(\tilde \mu, \tilde \nu) & \leq 2^{p-1} (4 \mathrm e (p+2))^{p+2} \int_{-\infty}^\infty \frac{|t|^p}{(2 \mathrm e (p+2) + |t|)^{p+2}} \, \mathrm{d}t \nonumber \\
& = \frac{2^p (4 \mathrm e (p+2))^{p+2}}{2 \mathrm e (p+1)(p+2)} \leq (c p)^p \label{eq:variance_term}
\,,
\end{align}
where $c$ is a universal constant.

Combining~\eqref{eq:bias_term} and~\eqref{eq:variance_term} with~\eqref{eq:wp_triangle} yields
\begin{equation*}
W_p(\mu, \nu) \lesssim p = p \left(\sup_{\ell \geq 1} \frac{\Delta_\ell}{\ell}\right)\,,
\end{equation*}
as claimed.
\end{proof}

\subsection{Proof of Lemma~\ref{lem:subexponential_moments}}
The claim is trivial if $\orlicz{X} = \infty$, so we assume $\orlicz{X} < \infty$, and indeed, by homogeneity, we may assume $\orlicz{X} = 1$.
We have
\begin{equation*}
\Big(\frac{|X|}{p}\Big)^p \leq (e^{|X|/p}-1)^p \leq e^{|X|} - 1 = \psi_1(X)\,,
\end{equation*}
so
\begin{equation*}
\E |X|^p \leq p^p \E \psi_1(X) \leq p^p\,.
\end{equation*}
\qed

\subsection{Proof of Lemma~\ref{lem:derivative_bound}}
\label{sec:lem:derivative_bound}
The analyticity of $f_m$ follows immediately from the well known fact that $\frac{\sin t}{t}$ is analytic, so it suffices to prove the derivative bound.
The claim will follow from the fact that
\begin{equation*}
\left|\frac{d^n}{dt^n} \left(\frac{\sin t}{t}\right)^{2m} \right| \leq \frac{2^{2m} (2m)^n}{(1+|t|)^{2m}}\,,
\end{equation*}
which we prove by induction on $m$.
Recall that, for function $f$ and $g$, the general Leibniz rule states
\begin{equation*}
\frac{d^n}{dt^n} f(t)g(t) =  \sum_{k=0}^n \binom{n}{k} f^{(n-k)}(t) g^{(k)}(t)\,.
\end{equation*}
We therefore have
\begin{equation*}
\left|\frac{d^n}{dt^n} \left(\frac{\sin t}{t}\right)^{2} \right| \leq \sum_{k=0} \binom{n}{k} \left|\frac{d^{n-k}}{dt^{n-k}} \frac{\sin t}{t}\right| \left|\frac{d^k}{dt^k} \frac{\sin t}{t}\right| = \frac{4}{(1+|t|)^2} \sum_{k=0}^n \binom{n}{k} = \frac{4 \cdot 2^n}{(1+|t|)^2}\,,
\end{equation*}
where we have used Lemma~\ref{lem:sinc_derivative_bound} to bound the derivatives of $\frac{\sin t}{t}$.
This proves the base case $m = 1$.
By induction, for $m > 1$, we have
\begin{align*}
\left|\frac{d^n}{dt^n} \left(\frac{\sin t}{t}\right)^{2 m} \right| & \leq \sum_{k=0}^n \binom{n}{k} \left|\frac{d^{n-k}}{dt^{n-k}} \left(\frac{\sin t}{t}\right)^2\right| \left|\frac{d^k}{dt^k} \left(\frac{\sin t}{t}\right)^{2m-2}\right| \\
& \leq \sum_{k=0}^n \binom{n}{k} \frac{4 \cdot 2^{n-k}}{(1+|t|)^2} \frac{2^{2m-2} (2m-2)^k}{(1+|t|)^{2m-2}} \\
& = \frac{2^{2m}}{(1+|t|)^{2m}} \sum_{k=0}^n \binom{n}{k} 2^{n-k}(2m-2)^k \\
& = \frac{2^{2m} (2m)^n}{(1+|t|)^{2m}}\,.
\end{align*}

The function $f_m(t)$ therefore satisfies
\begin{align*}
\left|\frac{d^n}{dx^n} f_m(t)\right| & = C_m \left|\frac{d^n}{dx^n} \left(\frac{\sin (t/4 \mathrm{e} m)}{t/4 \mathrm{e} m}\right)^{2 m}\right|  \leq C_m (4 \mathrm{e} m)^{-n} \frac{2^{2m} (2m)^n}{(1+|t/4 \mathrm{e} m|)^{2m}}  \leq \frac{(8 e m)^{2m}}{(2e)^n(4 \mathrm{e} m + |t|)^{2m}}\,,
\end{align*}
which concludes the proof.
\qed

\subsection{Proof of Proposition~\ref{proposition:moment_matching_tight}}
\label{SEC:proof:proposition:moment_matching_tight}
The first part is the content of Lemma~\ref{lem:priors}.
For the second part, for a given $\ep \in (0, 1/2]$, denote by $\cP_\ep$ the distribution on $\RR$ with density
\begin{equation*}
f_\ep(x) := c_\ep \mathrm{e}^{-|x|^{\frac{1}{1-\ep}}}\,,
\end{equation*}
where $c_\ep \leq 1$ is a suitable normalizing constant.
Note that the moment generating function of $\cP_\ep$ is finite everywhere.
Integrating $f_\ep$ implies that if $X \sim \cP_\ep$, then for all positive integers $p$,
\begin{equation*}
(\E |X|^p)^{1/p} \asymp p^{1-\ep}\,.
\end{equation*}
Denote by $\cP_\ep'$ the distribution of $2X$.
The coupling $(X, 2X)$ is a monotone coupling between $\cP_\ep$ and $\cP_\ep'$, so by~\cite[Theorem~2.9]{San15} we have
\begin{equation*}
W_p(\cP_\ep, \cP_\ep') = (\E |X - 2X|^p)^{1/p} \asymp p^{1-\ep}\,,
\end{equation*}
as claimed.
\qed

\subsection{Proof of Lemma~\ref{lem:moments_deconvolution}}
\label{SEC:proof:lem:moments_deconvolution}
We assume that $\orlicz{\cD} < \infty$, since otherwise the claim is vacuous.
Write $M_\mu$, $M_\nu$, and $M_\cD$ for the moment generating functions of $\mu$, $\nu$, and $\cD$, respectively.
We have
\begin{align*}
\Delta_\ell^\ell(\mu, \nu) & = \left|\frac{d^\ell}{dt^\ell}M_\mu(t) - M_\nu(t)\right|_{t = 0} \\
& = \left|\frac{d^\ell}{dt^\ell}\frac{M_\mu(t)M_\cD(t) - M_\nu(t)M_\cD(t)}{M_\cD(t)}\right|_{t = 0} \\
& \leq \sum_{m = 0}^\ell \binom{\ell}{m} \left|\frac{d^{\ell-m}}{dt^{\ell-m}}M_\mu(t)M_\cD(t) - M_\nu(t)M_\cD(t)\right|_{t = 0} \left|\frac{d^{m}}{dt^{m}}\frac{1}{M_\cD(t)}\right|_{t = 0} \\
& \leq  \left(\sum_{m = 0}^\ell \binom{\ell}{m}\left|\frac{d^{m}}{dt^{m}}\frac{1}{M_\cD(t)}\right|_{t = 0}\right) \cdot  \sup_{m \leq \ell} \Delta^m_m(\mu * \cD, \nu * \cD)
\end{align*}

If $X \sim \cD$, then
\begin{equation*}
|M_\cD(t) - 1| = | \E e^{tX} - 1| \leq \E |e^{tX} - 1| \leq \E e^{t |X|} - 1 = \E \psi_1(t |X|)\,,
\end{equation*}
which implies in particular that
\begin{equation*}
|M_\cD(t) - 1| \leq \frac 1 2 \quad \quad \forall t \leq (2\orlicz{\cD})^{-1}\,.
\end{equation*}

The function $\frac{1}{M_\cD(t)}$ is therefore analytic and bounded in norm by $2$ on a disk of radius $(2\orlicz{\cD})^{-1}$ around the origin.
Standard results from complex analysis (see, e.g.,~\cite{FlaSed09}, Proposition~IV.1) then imply that
\begin{equation*}
\left|\frac{d^{m}}{dt^{m}}\frac{1}{M_\cD(t)}\right|_{t = 0} \leq m! (2\orlicz{\cD})^m\,.
\end{equation*}
Combining this with the above bound yields
\begin{equation*}
\Delta_\ell^\ell(\mu, \nu) \leq \ell! (4 \orlicz{\cD})^\ell \cdot \sup_{m \leq \ell} \Delta^m_m(\mu * \cD, \nu * \cD)\,,
\end{equation*}
and the claim follows.
\qed

\section{Supplemental lemmas}\label{sec:lemmas}
\begin{lemma}\label{lem:disc_close}
For all $\nu \in \cM_{V}$, there exists a $\nu' \in \cM_{\cA, V}$ such that
\begin{equation*}
W_2(\nu, \nu') \leq (V+\sigma)n^{-1/4}\,.
\end{equation*}
\end{lemma}
\begin{proof}
Let $\Pi_{\cA, V}$ be the map sending each point in $[-V, V]$ to the nearest point in $\cA \cap [-V, V]$, and set $\nu' \defeq {\Pi_{\cA, V}}_\sharp \nu$.
Clearly $\nu' \in \cM_{\cA, V}$, and
\begin{equation*}
W_2^2(\nu, \nu') \leq \sup_{x \in [-V, V]} |x - \Pi_{\cA, V}(x)|^2 \leq \frac{(V+\sigma)^2}{n^{1/2}}\,,
\end{equation*}
which proves the claim.
\end{proof}

\begin{lemma}\label{lem:proj_close}
For all $\mu \in \cM_{V}$,
\begin{equation*}
W_2({\Pi_{\cA}}_\sharp (\mu * \cD), \mu * \cD) \lesssim (V+\sigma) n^{-1/4}\,.
\end{equation*}
\end{lemma}
\begin{proof}
By the definition of the Wasserstein distance, we have
\begin{equation*}
W_2^2({\Pi_{\cA}}_\sharp (\mu * \cD), \mu * \cD) \leq \E |\Pi_{\cA}(X + \xi) - X - \xi|^2 \quad \quad X \sim \mu, \xi \sim \cD\,.
\end{equation*}
If $|(X+\xi)| \leq (V + \sigma) \log n$, then $|\Pi_{\cA}(X + \xi) - X - \xi|^2 \leq (V + \sigma)^2n^{-1/2}$, which implies
\begin{align*}
\E |\Pi_{\cA}(X + \xi) - X - \xi|^2 & \leq (V+\sigma)n^{-1/4} + \E[|\alpha_N - X - \xi|^2 \1_{X + \xi > \alpha_N}] + \E[|\alpha_0 - X - \xi|^2 \1_{X + \xi < \alpha_0}] \\
& \leq (V+\sigma)^2n^{-1/2} + \E[|X + \xi|^2 \1_{|X + \xi| > (V + \sigma) \log n}] \\
& \leq (V+\sigma)^2n^{-1/2} + \E[|X + \xi|^4]^{1/2} \p[|X| > \sigma \log n]^{1/2}\,,
\end{align*}
where the last step uses the Cauchy-Schwarz inequality.

The assumption that $\orlicz{X} \leq \sigma$ implies
\begin{equation*}
\E[|X + \xi|^4]^{1/2} \lesssim (V+\sigma)^2
\end{equation*}
and
\begin{equation*}
\p[|X| > \sigma \log n] \leq \p[e^{|X|/\sigma} > n] \leq \frac{2}{n}\,.
\end{equation*}
Combining the above three displays yields
\begin{equation*}
\E |\Pi_{\cA}(X + \xi) - X - \xi|^2 \lesssim (V + \sigma)^2n^{-1/2}\,,
\end{equation*}
and this implies the stated bound.
\end{proof}

\begin{lemma}\label{lem:round_close}
Let $\mu$ be any measure on $[-V, V]$ with quantile function $\cQ_\mu$, and let $g \in \cF_V$ satisfy $g(x_i) = \cQ_\mu(i/n)$ for $1 \leq i \leq n$.
Then
\begin{equation*}
W_2(\mu, \pi_g) \leq 2V n^{-1/2}\,.
\end{equation*}
\end{lemma}
\begin{proof}
The definition of $\pi_g$ implies that the quantile function $\cQ_{\pi_g}$ of $\pi_g$ satisfies
\begin{equation*}
\cQ_{\pi_g}(x) = \cQ_\mu(i/n) \quad \quad \text{ where $(i-1)/n < x \leq i/n$}\,.
\end{equation*}
Since $\mu$ is supported on $[-V, V]$, we set $\cQ_\mu(0) \defeq \lim_{p \to O+} \cQ_\mu(p) \geq - V$.
By the explicit representation for the Wasserstein distance between one-dimensional measures~\cite[Theorem~2.10]{BobLed16}, we have
\begin{align*}
W_2^2(\mu, \pi_g) & = \int_{0}^1 |\cQ_\mu(x) - \cQ_{\pi_g}(x)|^2 \, \mathrm{d}x \\
& = \sum_{i=1}^n \int_{(i-1)/n}^{i/n} |\cQ_\mu(x) - \cQ_\mu(i/n)|^2 \, \mathrm{d}x \\
& \leq \sum_{i=1}^n \int_{(i-1)/n}^{i/n} |\cQ_\mu((i-1)/n) - \cQ_\mu(i/n)|^2\, \mathrm{d}x \\
& \leq \frac 1n \sum_{i=1}^n |\cQ_\mu((i-1)/n) - \cQ_\mu(i/n)|^2 \\
& \leq \frac{2V}{n} \sum_{i=1}^n \cQ_\mu(i/n) - \cQ_\mu((i-1)/n) \leq \frac{(2V)^2}{n}\,.
\end{align*}
\end{proof}

\begin{lemma}\label{lem:cm}
Let $f_m$ be defined as in~\eqref{eq:sinc_kernel}.
If $\int_{- \infty}^\infty f_m(t) \,\mathrm{d}t = 1$, then
\begin{equation*}
C_m \leq 1\,.
\end{equation*}
\end{lemma}
\begin{proof}
It suffices to show that $\int_{-\infty}^{\infty} \left(\frac{\sin(t/4 \mathrm{e} m)}{t/4 \mathrm{e}m}\right)^{2 m} \, \mathrm{d}t \geq 1$.
The inequality $|\sin(t)| \geq |t| - \frac{|t|^3}{6}$ implies
\begin{equation*}
\left(\frac{\sin(t/4 \mathrm{e}m)}{t/4 \mathrm{e}m}\right)^{2 m} \geq \left(1 - \frac{t^2}{6(4 \mathrm{e}m)^2}\right)^{2m} \geq 1- \frac{t^2}{48 \mathrm{e}^2 m} \geq \frac 12 \quad \text{if $t^2 \leq 24 \mathrm{e}^2 m$}\,.
\end{equation*}
Therefore
\begin{equation*}
\int_{-\infty}^{\infty} \left(\frac{\sin(t/4 \mathrm{e}m)}{t/4 \mathrm{e}m}\right)^{2 m} \, \mathrm{d}t \geq \int_{-1}^{1} \frac 12 \, \mathrm{d}t \geq 1\,.
\end{equation*}
\end{proof}

\begin{lemma}\label{lem:sinc_derivative_bound}
\begin{equation*}
\left|\frac{d^n}{dt^n} \frac{\sin t}{t}\right| \leq \frac{2}{1+|t|} \quad \quad \forall n \geq 0\,.
\end{equation*}
\end{lemma}
\begin{proof}
Recall that
\begin{equation*}
\frac{\sin t}{t} = \int_0^1 \cos(t x) \, \mathrm{d}x\,,
\end{equation*}
which implies after differentiating under the integral that
\begin{equation*}
\frac{d^n}{dt^n} \frac{\sin t}{t} = \int_{0}^1 x^n \cos^{(n)}(tx) \, \mathrm{d}x\,.
\end{equation*}
Since $|\cos^{(n)}(tx)| \leq 1$, we obtain immediately that
\begin{equation*}
\left| \frac{d^n}{dt^n} \frac{\sin t}{t} \right| \leq \int_0^1 x^n \, \mathrm{d}x = \frac{1}{n+1}\,,
\end{equation*}
which proves the claim when $|t| \leq 2n + 1$.

To prove the claim when $|t| > 2n + 1$, we proceed by induction.
When $n = 0$ and $|t| > 2n + 1 = 1$, the bound $|\sin(t)| \leq 1$ implies
\begin{equation*}
\left| \frac{\sin t}{t} \right| \leq \frac{1}{t} \leq \frac{2}{1 + |t|}\,.
\end{equation*}
We now assume that the bound in question holds for $n - 1$ and all $t$.
Integrating by parts and applying the induction hypothesis yields
\begin{align*}
\left|\int_{0}^1 x^n \cos^{(n)}(tx)\right| & = \left| \frac{\cos^{(n-1)} t}{t} - \frac{n}{t} \int_0^1 x^{n-1} \cos^{(n-1)}(tx) \, \mathrm{d}x\right| \\
& \leq \frac{1}{|t|} + \frac{n}{|t|} \frac{2}{1+|t|} = \frac{|t| + 2n + 1}{|t|(1+|t|)}\,.
\end{align*}
Since $|t| > 2n + 1$, this quantity is smaller than $\frac{2}{1 + |t|}$, as claimed.
\end{proof}

\begin{lemma}\label{lem:empirical_w2}
Let $\mu$ be any distribution satisfying $\orlicz{\mu} \leq K$, and let $\hat \mu = \frac 1n \sum_{i=1}^n \delta_{X_i}$, where $X_i \sim \mu$ are i.i.d.
Then
\begin{equation*}
\E[W_2^2(\mu, \hat \mu)] \leq \frac{16 K^2}{\sqrt n}\,.
\end{equation*}
\end{lemma}
\begin{proof}
We assume without loss of generality that $\mu$ is centered.
By~\cite[Theorem~7.16]{BobLed16},
\begin{equation*}
\E[W_2^2(\mu, \hat \mu)] \leq \frac{4}{\sqrt n} \int_{-\infty}^\infty |x| \sqrt{F(x)(1-F(x)} \, \mathrm{d}x\,,
\end{equation*}
where $F$ is the CDF of the measure $\mu$.
Let $X \sim \mu$. Then
\begin{equation*}
F(x)(1-F(x)) = \p[X \leq x]\p[X > x] \leq \p[|X| \geq |x|] \leq 2 e^{-|x|/K}\,,
\end{equation*}
where in the last step we have used the fact that $\orlicz{\mu} \leq K$.
We obtain
\begin{equation*}
\E[W_2^2(\mu, \hat \mu)] \leq \frac{16}{\sqrt n} \int_{0}^\infty x e^{-x / K} \, \mathrm{d}x = \frac{16 K^2}{\sqrt n}\,.
\end{equation*}
\end{proof}

\begin{lemma}\cite{CaiLow11}\label{lem:chi_square}
If $P$ and $Q$ are two centered measures supported on $[-V, V]$ such that $\Delta_\ell(P, Q) = 0$ for $\ell = 1, \dots, k-1$, then
\begin{equation*}
\mathrm{TV}((P * N(0,1))^{\otimes n}, (Q * N(0,1))^{\otimes n})^2 \leq \left(1 + e^{5V^2/2} \frac{(2V^2)^k}{k!}\right)^n - 1\,.
\end{equation*}
\end{lemma}
\begin{proof}
By~\cite{CaiLow11}, proof of Theorem~3, (see also~\cite{WuYan18}, Lemma 14), if $P$ and $Q$ are supported on $[-V, V]$, then
\begin{equation*}
\chi^2(P * N(0,1), Q * N(0,1)) \leq e^{V^2/2} \sum_{\ell = 1}^\infty \frac{\Delta_\ell^{2\ell}}{\ell!}\,.
\end{equation*}
By assumption, $\Delta_\ell(P, Q) = 0$ for $\ell < k$, and for $\ell \geq k$ the fact that $P$ and $Q$ are supported on $[-V, V]$ implies $\Delta_\ell^\ell \leq (2 V)^\ell$.
Combining these bounds yields
\begin{align*}
\chi^2(P * N(0,1), Q * N(0,1)) & \leq e^{V^2/2} \sum_{\ell \geq k} \frac{(2 V)^{2\ell}}{\ell!}  \leq e^{5 V^2/2} \frac{(2 V)^{2 k}}{k!}  \leq e^{5 V^2/2} \left(\frac{4eV^2}{k}\right)^k\,,
\end{align*}
where in the last step we have applied Stirling's approximation.

The claim then follows from standard properties of the $\chi^2$-divergence~\cite{Tsy09}.
\end{proof}

\bibliographystyle{amsalpha}
\bibliography{unlabeled_regression}

\address{{Philippe Rigollet}\\
{Department of Mathematics} \\
{Statistics and Data Science Center} \\
{Massachusetts Institute of Technology}\\
{77 Massachusetts Avenue,}\\
{Cambridge, MA 02139-4307, USA}\\
\printead{rigollet}
}

\address{{Jonathan Weed}\\
{Department of Mathematics} \\
{Statistics and Data Science Center} \\
{Massachusetts Institute of Technology}\\
{77 Massachusetts Avenue,}\\
{Cambridge, MA 02139-4307, USA}\\
\printead{weed}
}

\end{document}